\documentclass[12pt, reqno]{amsart}
\usepackage{amsmath, amsthm, amscd, amsfonts, amssymb, graphicx, color}
\usepackage[bookmarksnumbered, colorlinks, plainpages]{hyperref}
\hypersetup{colorlinks=true,linkcolor=blue, anchorcolor=blue, citecolor=blue, urlcolor=red, filecolor=magenta, pdftoolbar=true}

\newtheorem{theorem}{Theorem}[section]

\newtheorem{proposition}[theorem]{Proposition}
\newtheorem{corollary}[theorem]{Corollary}
\theoremstyle{definition}
\newtheorem{definition}[theorem]{Definition}

\newtheorem{property}[theorem]{Property}
\theoremstyle{remark}
\newtheorem{remark}[theorem]{Remark}
\numberwithin{equation}{section}

\begin{document}

\setcounter{page}{1}

\title[Maximum principle ...]{Maximum principle for space and time-space fractional partial differential equations}

\author[M. Kirane \MakeLowercase{and} B. T. Torebek]{Mokhtar Kirane \MakeLowercase{and} Berikbol T. Torebek$^{*}$}

\address{\textcolor[rgb]{0.00,0.00,0.84}{Mokhtar Kirane\newline LaSIE, Facult\'{e} des Sciences, \newline Pole Sciences et Technologies, Universit\'{e} de La Rochelle \newline Avenue M. Crepeau, 17042 La Rochelle Cedex, France \newline NAAM Research Group, Department of Mathematics, \newline Faculty of Science, King Abdulaziz University, \newline P.O. Box 80203, Jeddah 21589, Saudi Arabia}}
\email{\textcolor[rgb]{0.00,0.00,0.84}{mkirane@univ-lr.fr}}

\address{\textcolor[rgb]{0.00,0.00,0.84}{Berikbol T. Torebek \newline Department of Mathematics: Analysis,
Logic and Discrete Mathematics \newline Ghent University, Krijgslaan 281, Ghent, Belgium \newline Al--Farabi Kazakh National University \newline Al--Farabi ave. 71, 050040, Almaty, Kazakhstan \newline Institute of
Mathematics and Mathematical Modeling \newline 125 Pushkin str.,
050010 Almaty, Kazakhstan }}
\email{\textcolor[rgb]{0.00,0.00,0.84}{berikbol.torebek@ugent.be}}

\thanks{All authors contributed equally to the writing of this paper. All authors read and approved the final manuscript.}

\let\thefootnote\relax\footnote{$^{*}$Corresponding author}



\begin{abstract} In this paper we obtain new estimates of the sequential Caputo fractional derivatives of a function at its extremum points. We derive comparison principles for the linear fractional differential equations, and apply these principles to obtain lower and upper bounds of solutions of linear and nonlinear fractional differential equations. The extremum principle is then applied to show that the initial-boundary-value problem for nonlinear anomalous diffusion possesses at most one classical solution and this solution depends continuously  on the initial and boundary data. This answers positively to the open problem about maximum principle for the space and time-space fractional PDEs posed by Luchko in 2011. The extremum principle for an elliptic equation with a fractional derivative and for the fractional Laplace equation are also proved.
\end{abstract} \maketitle
\tableofcontents
\section{Introduction}
\subsection{Maximum principle for fractional derivatives} One of the most useful tools employed in the study of theory of partial differential equations are the maximum and minimum principles. They enable to obtain information about solutions without knowing their explicit forms.

Recently, with the development of fractional differential equations, the extremum principles for fractional differential equations have started to draw attention. This motivates us to consider the extremum principle for the sequential Caputo derivatives.

In \cite{Luchko1} Luchko proved a maximum principle for $\mathcal{D}_{0+}^\alpha$ Caputo fractional derivative in the following form:

\begin{itemize}
  \item Let  $f \in C^1((0, T )) \cap C([0, T ])$ attain its maximum over $[0, T ]$ at $t_0 \in (0, T ],$
then $\mathcal{D}_{0+}^\alpha f(t_0)\geq 0.$
  \item Let  $f \in C^1((0, T )) \cap C([0, T ])$ attain its minimum over $[0, T ]$ at $t_0 \in (0, T ],$
then $\mathcal{D}_{0+}^\alpha f(t_0)\leq 0.$
\end{itemize}

Based on an extremum principle for the Caputo fractional derivative he proved a maximum principle for the generalized time-fractional diffusion equation over an open bounded domain. The maximum principle was then applied to show some uniqueness and existence results for the initial-boundary-value problem for the generalized time-fractional diffusion equation \cite{Luchko2}. He also investigated the initial-boundary-value problems for the generalized multi-terms time-fractional diffusion equation \cite{Luchko3} and the diffusion equation of distributed order \cite{Luchko4}, and obtained some existence results for the generalized solutions in \cite{Luchko3, Luchko4}. For the one dimensional time-fractional diffusion equation, the generalized solution to the initial-boundary value problem was shown to be a solution in the classical sense \cite{Luchko5}.

In \cite{Al-Refai} Al-Refai generalized the results of Luchko as follows:

\begin{itemize}
  \item Let $f \in C^1([0, 1])$ attain its maximum at $t_0 \in (0, 1),$ then \begin{equation}\label{01}\mathcal{D}_{0+}^\alpha f(t_0) \geq \frac{{t_0}^{-\alpha}}{\Gamma(1-\alpha)}(f(t_0) - f(0)) \geq 0,\end{equation} for all $0 < \alpha < 1.$
  \item Let $f \in C^1([0, 1])$ attain its maximum at $t_0 \in (0, 1),$ then \begin{equation}\label{02}D_{0+}^\alpha f(t_0) \geq \frac{{t_0}^{-\alpha}}{\Gamma(1-\alpha)}f(t_0),\end{equation} for all $0 < \alpha < 1.$
Moreover, if $f(t_0) \geq 0,$ then $D_{0+}^\alpha f(t_0) \geq 0.$ Here $D_{0+}^\alpha$ is the Riemann-Liouville fractional derivative of order $\alpha.$
\end{itemize}

These results were used to prove the maximum principle for time-fractional diffusion equations by Al-Refai and Luchko \cite{Al-Refai-Luchko1, Al-Refai-Luchko2}, and in other papers \cite{Chan, LiuZeng, LuchkoYa, YeLiu}.

Maximum and minimum principles for time-fractional diffusion equations with singular kernel fractional derivative are proposed in \cite{CaoKong, KiraneTorebek}. In \cite{Al-Refai-Ab, Al-Refai2, KiraneBT, Borikhanov, Ahmad} it is proved a maximum principle for generalized time-fractional diffusion equations, based on an extremum principle for the fractional derivatives with non-singular kernel.

An investigation of the maximum principle for fractional Laplacian can be found in \cite{Cabre, Capella, Cheng, Del}.

As we have mentioned above, for classical fractional derivatives the question of the values of the fractional derivatives at the extreme points has been studied quite well for order of the derivatives in $(0,1).$ The case where the order of the derivative is less than two has not yet been fully investigated.

In \cite{ShiZhang} Shi and Zhang obtained the following results:
\begin{itemize}
  \item Let $f \in C^2(0, 1) \cap C([0, 1])$ attain its maximum at $t_0 \in (0, 1].$ Then $\mathcal{D}_{0+}^\alpha f(t_0) \leq 0,$ $1<\alpha\leq 2.$
\end{itemize}
But, Al-Refai \cite{Al-Refai} showed that the result claimed in \cite{ShiZhang} is not correct and he corrected the result of Shi and Zhang in the following form
\begin{itemize}
  \item Let $f \in C^2([0, 1])$ attain its minimum at $t_0 \in (0, 1),$ then
  $$\mathcal{D}_{0+}^\alpha f(t_0)\geq \frac{t_0^{-\alpha}}{\Gamma(2-\alpha)}\left[(\alpha-1)(f(0)-f(t_0))-t_0f'(0)\right],\,\, \textrm{for all} \,\, 1<\alpha<2;$$
  \item Let $f \in C^2([0, 1])$ attain its minimum at $t_0 \in (0, 1),$ and $f'(0) \leq 0,$ then $\mathcal{D}_{0+}^\alpha f(t_0) \geq 0,$ for all $1 < \alpha < 2;$
  \item Let $f \in C^2([0, 1])$ attain its minimum at $t_0 \in (0, 1),$ then $$D^\alpha_{0+} f(t_0)\geq \frac{t_0^{-\alpha}}{\Gamma(2-\alpha)}(\alpha-1)f(t_0),\,\, \textrm{for all} \,\, 1<\alpha<2.$$
  Moreover, if $f(t_0) \geq 0,$ then $D_{0+}^\alpha f(t_0) \geq 0.$
\end{itemize}

Luchko in \cite{Luchko5, Luchko6} proposed an important and interesting open problem: \textcolor{blue}{can one extend the maximum principle to the space and time-space fractional partial differential equations?.}

In \cite{YLAT14} Ye, Liu, Anh and Turner solved Luchko's open problem with additional conditions. More precisely, they proved a maximum principle for time-space fractional differential equations, when the derivatives of the solution at different ends of the interval have different signs. They considered time-space fractional differential equations with time-fractional Caputo derivative and the modified Riesz space-fractional derivatives in Caputo sense.

Further, in \cite{LiuZeng}, the above results were extended to time-space fractional differential equations with variable order. We also note that the maximum principle for the time-space fractional diffusion equation with space fractional Laplacian was studied in \cite{JiaLi}.

In this paper, we will give an answer to Luchko's open problem. We establish the maximum principle for the time-space and space fractional partial differential equations without additional conditions on the solution.

The aim  of this article is:
\begin{itemize}
  \item to obtain the estimates for the sequential Caputo fractional derivative of order $(1,2]$ at the extremum points;
  \item to obtain the comparison principles for the linear and non-linear ordinary fractional differential equations;
  \item to obtain the maximum and minimum principles for the time-space fractional diffusion and pseudo-parabolic equations with Caputo and Riemann-Liouville time-fractional derivatives;
  \item to obtain the uniqueness of solution and its continuous dependence on the initial conditions of the initial-boundary problems for the nonlinear time-space fractional diffusion and pseudo-parabolic equations;
  \item to obtain the maximum and minimum principles for the fractional elliptic equations;
  \item to obtain the uniqueness of solution of the boundary-value problems for the fractional elliptic equation.
\end{itemize}

\subsection{Definitions and some properties of fractional operators}
Let us give basic definitions of fractional differentiation and integration of the Riemann--Liouville and Caputo types.
\begin{definition}\cite{Kilbas} Let $f$ be an integrable real-valued function on the interval $[a, b],\, -\infty<a<b<+\infty$. The following integral
$$I^\alpha_{a+}  \left[ f \right]\left( t \right) = \left(f*K_{\alpha}\right)(t)={\rm{
}}\frac{1}{{\Gamma \left( \alpha \right)}}\int\limits_a^t {\left(
{t - s} \right)^{\alpha  - 1} f\left( s \right)} ds$$
is called the Riemann--Liouville integral operator of fractional order $\alpha>0.$ Here $K_{\alpha}=\frac{t^{\alpha-1}}{\Gamma(\alpha)},$ $\Gamma$ denotes the Euler gamma function.
\end{definition}
\begin{definition}\cite{Kilbas} Let $f\in L^1([a,b])$ and $f*K_{1-\alpha}\in W^{1,1}([a,b]),$ where $W^{1,1}([a,b])$ is the Sobolev space $$W^{1,1}([a,b])=\left\{f\in L^1([a,b]):\,\frac{d}{dt}f\in L^1([a,b])\right\}.$$ The Riemann--Liouville fractional derivative of order $0<\alpha<1$ of $f$ is defined as
$$D^\alpha_{a+} \left[ f \right](t) = \frac{{d}}{{dt }}I^{1 - \alpha }_{a+}\left[ f \right](t).$$
\end{definition}
\begin{definition}\label{def3}\cite{Kilbas} Let $f\in L^1([a,b])$ and $f*K_{1-\alpha}\in W^{1,1}([a,b]).$ For $0<\alpha<1,$ the fractional derivative of $f$
$$\mathcal{D}^\alpha_{a+} \left[ f \right]\left( t \right) = D^\alpha_{a+} \left[ {f\left( t \right) - f\left(a \right)}\right]$$
is the Caputo derivative.

If $f\in C^1([a,b])$, then the Caputo fractional derivative is defined as $$\mathcal{D}^\alpha_{a+} \left[ f \right]\left( t \right) = I^{1 - \alpha }_{a+} f'\left( t \right).$$
\end{definition}
Furthermore, we will use another kind of fractional order derivative. Namely, the sequential Caputo derivative of order $k\alpha, k = 1, 2, ...$
$$\mathcal{D}^{\alpha_1}_{a+}\cdot\mathcal{D}^{\alpha_2}_{a+}\cdot...\cdot\mathcal{D}^{\alpha_k}_{a+}.$$
Note that the concept of sequential derivative was introduced in \cite{Kilbas} (for more details we refer to \cite{MillerRoss, Podlubny, SKM87, TT16, TT14}). For the sequential derivative, in general
$\mathcal{D}^{\alpha_1}_{a+}\mathcal{D}^{\alpha_2}_{a+}\neq \mathcal{D}^{\alpha_1+\alpha_2}_{a+}.$

\begin{property}\label{p4}\cite{Kilbas}
If $f\in {{W}^{1}_2}\left([a, b] \right),$ then the Riemann-Liouville fractional derivative of order $\alpha $ can be represented in the form
$$D^{\alpha}_{a+} f\left( t \right)=\frac{f(a)}{\Gamma(1-\alpha)}\left(t-a\right)^{-\alpha}+\mathcal{D}^{\alpha}_{a+} f\left( t \right).$$
\end{property}

\begin{property}\label{p5} \cite{Kilbas} If $f\in L^1([a,b])$ and $I^{1-\alpha}f\in W^1_{2}([a,b]),$ then
$$I^{\alpha }_{a+}D^{\alpha }_{a+}f\left( t \right)=f(t)-I^{1-\alpha}_{a+}f(a)\frac{\left(t-a\right)^{\alpha-1}}{\Gamma(\alpha)}.$$
\end{property}

\begin{property}\label{p6} \cite{Kilbas} If $\beta>\alpha>0,\,$ and $-\infty<a<b<+\infty,$ then
$${D}^{\alpha}_{a+}\left( t-a \right)^{\beta-1}=\frac{\Gamma(\beta)}{\Gamma(\beta-\alpha)}\left( t-a \right)^{\beta-\alpha-1}.$$
\end{property}

\begin{property}\label{p7} \cite{Kilbas} If $0<\alpha<1$ and $-\infty<a<b<+\infty,$ then
$${D}_{a+}^{\alpha }\left( t-a \right)^{\alpha-1}=0.$$
\end{property}

\section{Main Results}\label{MR}

\begin{proposition}\label{prop1} Let $f\in C^1([a,b])$ and $\mathcal{D}^{\alpha}_{a+}\mathcal{D}^{\beta}_{a+}f\in C([a,b]).$ If $f$ attains its minimum value at $x^*\in \left[a, b \right]$, then\\
(i) if $0<\alpha, \beta<1$ and $1<\alpha+\beta<2,$
\begin{equation}\label{0.1}\mathcal{D}^{\alpha}_{a+}\mathcal{D}^{\beta}_{a+} f\left(x^*\right)\geq \frac{\alpha+\beta-1}{\Gamma(2-\alpha-\beta) }\left(x^* - a\right)^{-\alpha-\beta}\left( f\left(a\right)-f\left(x^*\right) \right)\geq 0;
\end{equation}
(ii)	if $0<\alpha, \beta<1$ and $\alpha+\beta<1,$
\begin{equation}\label{0.1*}\mathcal{D}^{\alpha}_{a+}\mathcal{D}^{\beta}_{a+} f\left(x^*\right)\leq \frac{\alpha+\beta-1}{\Gamma(2-\alpha-\beta) }\left(x^* - a\right)^{-\alpha-\beta}\left( f\left(a\right)-f\left(x^*\right) \right)\leq 0;
\end{equation}
(iii)	if $0<\alpha, \beta<1$ and $\alpha+\beta=1,$ $\mathcal{D}^{\alpha}_{a+}\mathcal{D}^{\beta}_{a+} f\left(x^*\right)=0.$
\end{proposition}

\begin{proof}
For the proof of part (i), we define, as in \cite{ShiZhang} the auxiliary function $$h\left( x \right)=f\left( x \right)-f\left( x^* \right),\text{ }x\in \left[a,b \right];$$ then it follows that $h\left( x \right)\ge 0,$ on $\left[a, b \right],$ $h\left(x^*\right)=h'\left(x^*\right)=0$ and $$\mathcal{D}^{\alpha}_{a+}\mathcal{D}^{\beta}_{a+} h \left( x \right)=\mathcal{D}^{\alpha}_{a+}\mathcal{D}^{\beta}_{a+}f\left( x \right).$$
Since $h\in C^1([a,b])$ and $\mathcal{D}^{\alpha}_{a+}\mathcal{D}^{\beta}_{a+}h\left(x \right)\in C([a,b]),$ we have
\begin{align*}\mathcal{D}^{\alpha}_{a+}\mathcal{D}^{\beta}_{a+} h \left( x^* \right)&
=\frac{1}{\Gamma(1-\alpha)}\int\limits_a^{x^*}\left(x^*-s\right)^{-\alpha} \frac{d}{ds}\mathcal{D}^{\beta}_{a+} h \left( s \right)ds\\&=
\frac{1}{\Gamma(2-\alpha)}\frac{d}{dx^*}\int\limits_a^{x^*}\left(x^*-s\right)^{1-\alpha} \frac{d}{ds}\mathcal{D}^{\beta}_{a+} h \left( s \right)ds\\&=
\frac{1}{\Gamma(2-\alpha)}\frac{d}{dx^*}\left[(x^*-s)^{1-\alpha}\mathcal{D}^{\beta}_{a+} h \left( s \right)\Big{|}_{s=a}^{s=x^*} + (1-\alpha)\int\limits_a^{x^*}\left(x^*-s\right)^{-\alpha} \mathcal{D}^{\beta}_{a+} h \left( s \right)ds\right] \\&=
-\frac{(x^*-a)^{-\alpha}}{\Gamma(1-\alpha)}\mathcal{D}^{\beta}_{a+} h \left( a \right) + \frac{d}{dx^*}\int\limits_a^{x^*}h'(\tau)\int\limits_{\tau}^{x^*}\frac{\left(x^*-s\right)^{-\alpha} (s-\tau)^{-\beta}}{\Gamma(1-\alpha) \Gamma(1-\beta)} ds d\tau.\end{align*}
As $h\in C^1([a,b]),$ then
\begin{align*}\left|\mathcal{D}^{\beta}_{a+} h \left( x \right)\right|&\leq\frac{1}{\Gamma(1-\beta)}\int\limits_a^x (x-s)^{-\beta} |h'(s)| ds\\& \leq \sup\limits_{a\leq s\leq x}|h'(s)|\frac{1}{\Gamma(2-\beta)}(x-s)^{1-\beta}\Big{|}_a^x;\end{align*}
therefore $\mathcal{D}^{\beta}_{a+} h \left( a \right)=0.$

Then, for $\mathcal{D}^{\alpha}_{a+}\mathcal{D}^{\beta}_{a+} h $ we obtain \begin{equation}\label{0.2}\mathcal{D}^{\alpha}_{a+}\mathcal{D}^{\beta}_{a+} h \left( x^* \right)=\frac{d}{dx^*}\int\limits_a^{x^*}h'(\tau)\int\limits_{\tau}^{x^*}\frac{\left(x^*-s\right)^{-\alpha} (s-\tau)^{-\beta}}{\Gamma(1-\alpha) \Gamma(1-\beta)} ds d\tau.\end{equation}
As
$$\int\limits_{\tau}^{x^*}\left(x^*-s\right)^{-\alpha} (s-\tau)^{-\beta}ds=\frac{\Gamma(1-\alpha)\Gamma(1-\beta)}{\Gamma(2-\alpha-\beta)}(x^*-\tau)^{1-\alpha-\beta},$$
hence
\begin{align*}\mathcal{D}^{\alpha}_{a+}\mathcal{D}^{\beta}_{a+} h \left( x^* \right)&=\frac{1}{\Gamma(2-\alpha-\beta)}\frac{d}{dx^*}\int\limits_a^{x^*}(x^*-\tau)^{1-\alpha-\beta} h'(\tau) d\tau\\& =\frac{1-\alpha-\beta}{\Gamma(2-\alpha-\beta)}\int\limits_a^{x^*}(x^*-\tau)^{-\alpha-\beta} h'(\tau) d\tau.\end{align*}
Integrating by parts yields
\begin{align*}\mathcal{D}^{\alpha}_{a+}\mathcal{D}^{\beta}_{a+} h \left( x^* \right)&=\frac{1-\alpha-\beta}{\Gamma(2-\alpha-\beta)}\int\limits_a^{x^*}(x^*-\tau)^{-\alpha-\beta} h'(\tau) d\tau \\&=\frac{1-\alpha-\beta}{\Gamma(2-\alpha-\beta)} (x^*-\tau)^{-\alpha-\beta}h(\tau)\Big{|}_a^{x^*}\\& -\frac{(1-\alpha-\beta)(\alpha+\beta)}{\Gamma(2-\alpha-\beta)}\int\limits_a^{x^*}(x^*-\tau)^{-\alpha-\beta-1}h(\tau)d\tau.\end{align*}
As $h\left(x^*\right)=h'\left(x^*\right)=0,$ then using l'Hosptial's rule,
$$\lim\limits_{\tau\rightarrow x^*}(x^*-\tau)^{-\alpha-\beta}h(\tau)=0.$$ Whereupon \begin{align*}\mathcal{D}^{\alpha}_{a+}\mathcal{D}^{\beta}_{a+} h \left( x^* \right)&=\frac{\alpha+\beta-1}{\Gamma(2-\alpha-\beta)}(x^*-a)^{-\alpha-\beta}h(a)\\& +\frac{(\alpha+\beta-1)(\alpha+\beta)}{\Gamma(2-\alpha-\beta)}\int\limits_a^{x^*}(x^*-\tau)^{-\alpha-\beta-1}h(\tau)d\tau.\end{align*}
Now, if $\alpha+\beta>1,$ we have $$\frac{(\alpha+\beta-1)(\alpha+\beta)}{\Gamma(2-\alpha-\beta)}\int\limits_a^{x^*}(x^*-\tau)^{-\alpha-\beta-1}h(\tau)d\tau\geq 0.$$ Whence
\begin{align*}\mathcal{D}^{\alpha}_{a+}\mathcal{D}^{\beta}_{a+} h \left( x^* \right)&\geq\frac{\alpha+\beta-1}{\Gamma(2-\alpha-\beta)}(x^*-a)^{-\alpha-\beta}h(a)\\& = \frac{\alpha+\beta-1}{\Gamma(2-\alpha-\beta)}(x^*-a)^{-\alpha-\beta}\left(f(a)-f(x^*)\right)\geq 0.\end{align*}
If $\alpha+\beta<1,$ we have $$\frac{(\alpha+\beta-1)(\alpha+\beta)}{\Gamma(2-\alpha-\beta)}\int\limits_a^{x^*}(x^*-\tau)^{-\alpha-\beta-1}h(\tau)d\tau\leq 0.$$ Consequently,
\begin{align*}\mathcal{D}^{\alpha}_{a+}\mathcal{D}^{\beta}_{a+} h \left( x^* \right)&\leq-\frac{1-\alpha-\beta}{\Gamma(2-\alpha-\beta)}(x^*-a)^{-\alpha-\beta}h(a)\\& = \frac{1-\alpha-\beta}{\Gamma(2-\alpha-\beta)}(x^*-a)^{-\alpha-\beta}\left(f(x^*)-f(a)\right)\leq 0.\end{align*}
When $\alpha+\beta=1,$ we have $\mathcal{D}^{\alpha}_{a+}\mathcal{D}^{\beta}_{a+} h \left( x^* \right)=\mathcal{D}^{\alpha}_{a+}\mathcal{D}^{\beta}_{a+} f \left( x^* \right)=0.$
\end{proof}
\begin{proposition}\label{prop2} Let $f\in C^1([a,b])$ and $\mathcal{D}^{\alpha}_{a+}\mathcal{D}^{\beta}_{a+}f\in C([a,b]).$ If $f$ attains its maximum at the point $x^*\in \left[a, b \right]$, then\\
(i) if $0<\alpha, \beta<1$ and $1<\alpha+\beta<2,$ then
\begin{equation}\label{0.01}\mathcal{D}^{\alpha}_{a+}\mathcal{D}^{\beta}_{a+} f\left(x^*\right)\leq \frac{\alpha+\beta-1}{\Gamma(2-\alpha-\beta) }\left(x^* - a\right)^{-\alpha-\beta}\left( f\left(a\right)-f\left(x^*\right) \right)\leq 0.
\end{equation}
(ii)	if $0<\alpha, \beta<1$ and $\alpha+\beta<1,$ then
\begin{equation}\label{0.02}\mathcal{D}^{\alpha}_{a+}\mathcal{D}^{\beta}_{a+} f\left(x^*\right)\geq \frac{\alpha+\beta-1}{\Gamma(2-\alpha-\beta) }\left(x^* - a\right)^{-\alpha-\beta}\left( f\left(a\right)-f\left(x^*\right) \right)\geq 0.
\end{equation}
(iii)	if $0<\alpha, \beta<1$ and $\alpha+\beta=1,$ then $\mathcal{D}^{\alpha}_{a+}\mathcal{D}^{\beta}_{a+} f\left(x^*\right)=0.$
\end{proposition}

\section{Comparison principle for ordinary fractional differential equations}
We use the results in Section \ref{MR} to obtain new comparison principles for
linear and nonlinear fractional differential equation.
\subsection{Linear fractional differential equations}
\begin{theorem}\label{ODE1} Let a function $u\in C^1([a,b])$ be such that $\mathcal{D}^{\alpha}_{a+}\mathcal{D}^{\beta}_{a+}u\in C([a,b])$  satisfy
the equation
\begin{equation}\label{ode1}\mathcal{D}^{\alpha}_{a+}\mathcal{D}^{\beta}_{a+}u(x)+q(x)u(x)=f(x),\, x\in (a,b),\end{equation}
where $0<\alpha, \beta<1, \,\,1<\alpha+\beta\leq 2,$ $q(x)\leq 0$ is continuous on $[a, b]$ and $q(a)\neq 0.$ If $f(x)\geq 0,\, x>a,$ then $u(x)\leq 0$ for $x\in [a,b].$\end{theorem}
\begin{proof} Since $u\in C^1([a,b])$ and $\mathcal{D}^{\alpha}_{a+}\mathcal{D}^{\beta}_{a+}u\in C([a,b]),$ then
\begin{align*}\left|\mathcal{D}^{\alpha}_{a+}\mathcal{D}^{\beta}_{a+} u \left( x \right)\right|&\leq\frac{1}{\Gamma(1-\alpha)}\lim\limits_{x\rightarrow a}\int\limits_a^x (x-s)^{-\alpha} \left|\frac{d}{ds}\mathcal{D}^{\beta}u(s)\right| ds\\& \leq \frac{\max\limits_{a\leq s\leq x}|\mathcal{D}^{\beta}u(s)|}{\Gamma(2-\alpha)}\lim\limits_{x\rightarrow a}(x-s)^{1-\alpha}\Big{|}_a^x.\end{align*}
Therefore $\mathcal{D}^{\alpha}_{a+}\mathcal{D}^{\beta}_{a+} u \left( a \right)=0.$  By the continuity of the solution and $f\geq 0,$ we have $u(a)\leq 0.$  Assume by contradiction that the result is not true, as $u$ is continuous on $[a, b]$ then $u$ attains its maximum $u(x^*)>0$ at $x^*> a.$  Applying the result of Proposition \ref{prop1} and inequality \eqref{01} we obtain
$$\mathcal{D}^{\alpha}_{a+}\mathcal{D}^{\beta}_{a+} u \left( x^* \right)+q(x^*)u(x^*)< 0,$$ which contradicts \begin{equation*}\mathcal{D}^{\alpha}_{a+}\mathcal{D}^{\beta}_{a+}u(x)+q(x)u(x)\geq 0.\end{equation*} The proof is complete.
\end{proof}
\begin{theorem} Let $u_1, u_2\in C^1([a,b])$ be such that $\mathcal{D}^{\alpha}_{a+}\mathcal{D}^{\beta}_{a+}u_1, \mathcal{D}^{\alpha}_{a+}\mathcal{D}^{\beta}_{a+}u_2\in C([a,b])$ solutions of
$$\mathcal{D}^{\alpha}_{a+}\mathcal{D}^{\beta}_{a+}u_1(x)+q(x)u_1(x)=f_1(x), x\in (a,b),$$
$$\mathcal{D}^{\alpha}_{a+}\mathcal{D}^{\beta}_{a+}u_2(x)+q(x)u_2(x)=f_2(x), x\in (a,b),$$
where $0<\alpha, \beta<1,$ $1<\alpha+\beta<2,$ $q(x)\leq 0,\, f_1(x), f_2(x)$ are continuous on $[a, b]$ and $q(a)\neq 0.$ If $f_1(x)\leq f_2(x),$ then $u_1(x)\leq u_2(x),\,x\in [a,b].$
\end{theorem}
\begin{proof} The function $\tilde{u}=u_1-u_2$ satisfies $$\mathcal{D}^{\alpha}_{a+}\mathcal{D}^{\beta}_{a+}\tilde{u}(x)+q(x)\tilde{u}(x)=f_1(x)-f_2(x)\leq 0, x\in (a,b).$$
By Theorem \ref{ODE1}, we have $\tilde{u}(x)\leq 0,$ and whereupon $u_1(x)\leq u_2(x),\,x\in [a,b].$
\end{proof}
\subsection{Non-linear fractional differential equations}
\begin{theorem} Let a function $u\in C^1([a,b])$ be such that $\mathcal{D}^{\alpha_j}_{a+}\mathcal{D}^{\beta_j}_{a+}u\in C([a,b]), j=1,...,m$  satisfy
the following equation
\begin{equation*} \sum\limits_{j=1}^m\lambda_j\mathcal{D}^{\alpha_j}_{a+}\mathcal{D}^{\beta_j}_{a+}u(x)+F(x,u)=0,\, x\in(a,b),\end{equation*}
where $0\leq \lambda_j\in \mathbb{R},  0<\alpha_j,\beta_j\leq 1,\,1<\alpha_j+\beta_j\leq 2,\, j=1,...,m,$ $F(x, u)$ is a smooth function. If $F(x, u)$ is non-increasing with respect to $u,$
then the above equation has at most one solution.
\end{theorem}
\begin{proof}
Let $u_1$ and $u_2$ be two solutions of the above equation and let $\hat{u}=u_1 - u_2,$ then
$$\sum\limits_{j=1}^m\lambda_j\mathcal{D}^{\alpha_j}_{a+}\mathcal{D}^{\beta_j}_{a+}\hat{u}(x)+F(x,u_1)-F(x,u_2)=0.$$
Applying the mean value theorem we obtain
$$F(x,u_1)-F(x,u_2)=\frac{\partial F}{\partial u}(u^*)(u_1-u_2),$$ for some $u^*$ between $u_1$ and $u_2.$ Thus,
$$\sum\limits_{j=1}^m\lambda_j\mathcal{D}^{\alpha_j}_{a+}\mathcal{D}^{\beta_j}_{a+}\hat{u}(x)=-\frac{\partial F}{\partial u}(u^*)\hat{u}.$$
Assume by contradiction that $\hat{u}$ is not identically zero. Then $\hat{u}$ has either a positive maximum or a negative minimum. At a positive maximum $x^*\in[a,b],$ and as $F(x,u)$
is non-increasing, we have $$\frac{\partial F}{\partial u}(u^*)<0\,\,\, \textrm{and}\,\,\, \frac{\partial F}{\partial u}(u^*)\hat{u}<0;$$ then
$\sum\limits_{j=1}^m\lambda_j\mathcal{D}^{\alpha_j}_{a+}\mathcal{D}^{\beta_j}_{a+}\hat{u}(x^*)\leq 0.$ By using results of Proposition \ref{prop1} for a positive maximum and a negative minimum, respectively, we obtain $\hat{u}=0.$
\end{proof}
\begin{theorem} Consider the nonlinear equation
\begin{equation}\label{ode2}\mathcal{D}^\alpha_{a+}\mathcal{D}^\alpha_{a+} u(x)+F(x,u)=0,\, x\in(a,b),\end{equation} where $0<\alpha, \beta\leq 1,\, 1<\alpha+\beta\leq 2,\, 0<\gamma\leq 1,$ $F(x, u)$ is a smooth function. Suppose that
$$\mu_2u+q_2(x)\leq F(x,u)\leq \mu_1u+q_1(x),\,x\in(a,b),$$ where $\mu_1,\mu_2<0.$  Let $u_1$ and $u_2$ be the solutions of
\begin{equation}\label{ode3}\mathcal{D}^\alpha_{a+}\mathcal{D}^\alpha_{a+} u_1(x)-\mu_1u_1(x)-q_1(x)=0,\, x\in(a,b),\end{equation} and $$\mathcal{D}^\alpha_{a+}\mathcal{D}^\alpha_{a+} u_2(x)-\mu_2u_2(x)-q_2(x)=0,\, x\in(a,b).$$ Then $u_2(x)\leq u(x)\leq u_1(x),\,x\in[a,b].$
\end{theorem}
\begin{proof}We shall prove that $u(x) \leq u_1(x)$ and by applying analogous steps one can show that $u(x) \geq u_2(x).$ By subtracting \eqref{ode3} from \eqref{ode2} we have
\begin{align*}\mathcal{D}^\alpha_{a+}\mathcal{D}^\alpha_{a+} (u(x)-u_1(x))=-F(x,u)+\mu_1u_1(x)+q_1(x) \geq -\mu_1(u(x)-u_1(x)).\end{align*}
Then \begin{align*}\mathcal{D}^\alpha_{a+}\mathcal{D}^\alpha_{a+} \tilde{u}(x)+\mu_1 \tilde{u}(x)\geq 0\end{align*} for $ \tilde{u}(x)=u(x)-u_1(x).$
Since $\mu_1 < 0,$ it follows that by Theorem \ref{ode1}, $\tilde{u}\leq 0,$ the proof is complete.
\end{proof}
\section{Time-space fractional diffusion equation with Caputo derivative}
In  this section, we  consider the nonlinear time-space fractional diffusion equation with Caputo derivative
\begin{equation}\label{1.1}\mathcal{D}_{0+,t}^{\alpha}u(x,t)=\nu\mathcal{D}_{a+,x}^{\beta_1}\mathcal{D}_{a+,x}^{\beta_2} u(x,t)+F\left( x,t, u\right),\, (x,t)\in (a,b)\times \left(0,T \right]=\Omega,\end{equation}
subject to the initial condition
\begin{equation}\label{1.2}u\left( x, 0\right)=\varphi ( x ),\,x \in [a,b],\end{equation}
and boundary conditions
\begin{equation}\label{*1.2}u\left( a, t\right)=\phi(t),\,u\left( b, t\right)=\psi(t),\,t\geq 0,\end{equation}
where $\alpha\in (0,1],\, 0<\beta_1, \beta_2\leq 1, 1<\beta_1+\beta_2\leq 2,$ $\nu> 0,$ $-\infty<a<b<+\infty,$ and the functions $F,\varphi, \phi, \psi$ are continuous.

\subsection{Maximum principle}
In this subsection, we shall present the maximum (minimum) principle for the linear case of equation \eqref{1.1}.

\begin{theorem}\label{t1} Let $u\in C([a,b]\times[0,T])$ and let $\mathcal{D}_{0+,t}^{\alpha}u, \mathcal{D}_{a+,x}^{\beta_1}\mathcal{D}_{a+,x}^{\beta_2} u\in C([a,b]\times[0,T]).$ Let $u\left( x,t \right)$  satisfy \begin{equation}\label{L_eq} \mathcal{D}_{0+,t}^{\alpha}u(x,t)=\nu\mathcal{D}_{a+,x}^{\beta_1}\mathcal{D}_{a+,x}^{\beta_2} u(x,t)+F\left(x,t\right),\, (x,t)\in (a,b)\times \left(0,T \right],\end{equation} with initial-boundary conditions \eqref{1.2}, \eqref{*1.2}. If $F\left( x,t \right)\geq 0$ for $\left( x,t \right)\in \overline{\Omega },$ then $$u\left( x,t \right)\ge \underset{\left( x,t \right)\in \overline{\Omega}}{\mathop{\min }}\,\{\phi\left(t \right), \psi\left(t \right),\varphi \left( x \right)\}\text{ for }\left( x,t \right)\in \overline{\Omega }.$$
\end{theorem}
\begin{proof}Let $m=\underset{\left( x,t \right)\in \overline{\Omega }}{\mathop{\min }}\,\{\phi\left(t \right), \psi\left(t \right), \varphi \left( x \right)\}\text{ }$ and $\tilde{u}\left( x,t \right)=u\left( x,t \right)-m.$ Then, from \eqref{1.2}-\eqref{*1.2}, we obtain \begin{align*}&\tilde{u}\left(a, t \right)=\phi\left(t \right)-m\ge 0,\,\,t\in \left[0, T\right],\\& \tilde{u}\left(b, t \right)=\psi \left(t \right)-m\ge 0,\,\,t\in \left[0, T\right],\end{align*} and $$\tilde{u}\left(x,0 \right)=\varphi \left( x \right)-m\ge 0,\,\,x\in [a,b].$$

Since $$\mathcal{D}_{0+,t}^{\alpha}\tilde{u}(x,t)=\mathcal{D}_{0+,t}^{\alpha}u(x,t)$$ and $$\mathcal{D}_{a+,x}^{\beta_1}\mathcal{D}_{a+,x}^{\beta_2}\tilde{u}\left( x,t \right)=\mathcal{D}_{a+,x}^{\beta_1}\mathcal{D}_{a+,x}^{\beta_2} u\left( x,t \right),$$
it follows that $\tilde{u}\left( x,t \right)$ satisfies \eqref{L_eq}:
$$\mathcal{D}_{0+,t}^{\alpha}\tilde{u}(x,t)=\nu\mathcal{D}_{a+,x}^{\beta_1}\mathcal{D}_{a+,x}^{\beta_2}\tilde{u}\left( x,t \right)+F\left(x,t \right),$$
and initial condition
$$\tilde{u}\left(x, 0\right)=\varphi \left(x\right)-m\geq 0,\,x\in [a,b].$$
Suppose that there exits some $\left( x,t \right) \in \overline{\Omega }$  such that $\tilde{u}\left( x,t \right)<0.$
Since $$\tilde{u}\left( x,t \right)\ge 0,\,\,\left( x,t \right)\in \{a\}\times \left[0, T \right]\cup\{b\}\times \left[0, T \right]\cup \bar (a,b)\times \{0\},$$
there is $\left( {{x}_{0}},{{t}_{0}} \right) \in \Omega$  such that $\tilde{u}\left( {{x}_{0}},{{t}_{0}} \right)$ is a negative minimum of
$\tilde{u}$ over $\Omega.$ It follows from Proposition \ref{prop1} that

\begin{equation*}\mathcal{D}_{a+,x}^{\beta_1}\mathcal{D}_{a+,x}^{\beta_2} \tilde{u}\left(x_0,t_0\right)\geq \frac{\beta_1+\beta_2-1}{\Gamma(2-\beta_1-\beta_2) }\left(x_0 - a\right)^{-\beta_1-\beta_2}\left( \tilde{u}\left(a,t_0\right)-\tilde{u}\left(x_0,t_0\right) \right)\geq 0.\end{equation*}

From the results of Al-Refai and Luchko \cite{Al-Refai-Luchko1} it follows that $$\mathcal{D}_{0+,t}^\alpha \tilde{u}(x_0, t_0) \leq \frac{{t_0}^{-\alpha}}{\Gamma(1-\alpha)}(\tilde{u}(x_0, t_0) - \varphi(x_0)+m) < 0,$$

Therefore at $\left( {{x}_{0}},{{t}_{0}} \right)$, we get $$\mathcal{D}_{0+,t}^\alpha \tilde{u}(x_0, t_0)< 0$$ and $$\nu\mathcal{D}_{a+,x}^{\beta_1}\mathcal{D}_{a+,x}^{\beta_2} \tilde{u}\left(x_0,t_0\right)+F\left( {{x}_{0}},{{t}_{0}} \right)\ge 0.$$ This contradiction shows that $\tilde{u}\left( x,t \right)\ge 0$ on $\overline{\Omega },$ whereupon  $u\left( x,t \right)\ge m$ on $\overline{\Omega }$.
\end{proof}

\begin{theorem}\label{t2} Let $u\in C([a,b]\times[0,T])$ and let $\mathcal{D}_{0+,t}^{\alpha}u, \mathcal{D}_{a+,x}^{\beta_1}\mathcal{D}_{a+,x}^{\beta_2} u\in C([a,b]\times[0,T]).$
Suppose that $u\left( x,t \right)$  is the solution of the problem \eqref{L_eq}, \eqref{1.2}, \eqref{*1.2}. If $F\left( x,t \right)\le 0$ for $\left( x,t \right)\in \overline{\Omega },$ then $$u\left( x,t \right)\leq \underset{\left( x,t \right)\in \overline{\Omega}}{\mathop{\max }}\,\{\phi\left(t \right), \psi\left(t \right),\varphi \left( x \right)\}\text{ for }\left( x,t \right)\in \overline{\Omega }.$$ \end{theorem}
Theorem \ref{t1} and \ref{t2} imply the following assertions.
\begin{corollary} Let $u\in C([a,b]\times[0,T])$ and let $\mathcal{D}_{0+,t}^{\alpha}u, \mathcal{D}_{a+,x}^{\beta_1}\mathcal{D}_{a+,x}^{\beta_2} u\in C([a,b]\times[0,T]).$ Suppose that $u\left( x,t \right)$  satisfy \eqref{L_eq}, \eqref{1.2} and \eqref{*1.2}.  If $F\left( x,t \right)\geq 0$ for $\left( x,t \right)\in \overline{\Omega },$ $\varphi\left(x\right)\geq 0,\,x\in[a,b],$ and $\phi\left(t \right)\geq 0, \psi\left(t \right)\geq 0, \, t\in[0,T],$ then $$u\left( x,t \right)\geq 0,\,\,\left( x,t \right)\in \overline{\Omega }.$$
\end{corollary}
\begin{corollary} Let $u\in C([a,b]\times[0,T])$ and let $\mathcal{D}_{0+,t}^{\alpha}u, \mathcal{D}_{a+,x}^{\beta_1}\mathcal{D}_{a+,x}^{\beta_2} u\in C([a,b]\times[0,T]).$ Suppose that $u\left( x,t \right)$  satisfies \eqref{L_eq}, \eqref{1.2} and \eqref{*1.2}. If $F\left( x,t \right)\leq 0$ for $\left( x,t \right)\in \overline{\Omega },$ $\varphi\left(x\right)\leq 0,\,x\in[a,b],$ and $\phi\left(t \right)\leq 0, \psi\left( t \right)\leq 0,\,t\in[0,T],$  then $$u\left( x,t \right)\leq 0,\,\,\left( x,t \right)\in \overline{\Omega }.$$
\end{corollary}
Theorems \ref{t1} and \ref{t2} are similar to the weak maximum principle for the heat equation.

\subsection{Linear time-space fractional diffusion equation}\label{subsection2}
Analogue to the classical case, the fractional version of the weak maximum principle can be used to prove the uniqueness of a solution.
\begin{theorem}\label{t3}
The problem \eqref{L_eq}, \eqref{1.2} with Dirichlet boundary conditions \begin{equation}\label{Dir1}u(a,t)=0,\,u(b,t)=0,\, t\in [0,T]\end{equation} has at most one solution.
\end{theorem}
\begin{proof}
Let ${{u}_{1}}\left( x,t \right)$ and ${{u}_{2}}\left( x,t \right)$ be two different solutions of the problem \eqref{1.1}-\eqref{1.2}, and let $u=u_1-u_2,$ then,
$$\mathcal{D}_{0+,t}^{\alpha }u\left( x,t \right)=\nu\mathcal{D}_{a+,x}^{\beta_1}\mathcal{D}_{a+,x}^{\beta_2}u\left( x,t \right),$$
with zero initial and boundary conditions. It follows from Theorems \ref{t1} and \ref{t2} that $${{u}_{1}}\left( x,t \right)-{{u}_{2}}\left( x,t \right)\equiv 0,\,\, \textrm{on} \,\,\overline{\Omega }.$$ A contradiction. The result then follows.\end{proof}
Theorems \ref{t1} and \ref{t2} can be used to show that a solution $u\left( x,t \right)$ of the problem \eqref{L_eq}, \eqref{1.2}, \eqref{Dir1} depends continuously on the initial data $\varphi \left( x \right).$

\begin{theorem}\label{t4}
Suppose that, $u\left( x,t \right)$ and $\overline{u}\left( x,t \right)$ are the solutions of the problem \eqref{L_eq}, \eqref{1.2}, \eqref{Dir1} corresponding to the initial data $\varphi \left( x \right)$ and $\overline{\varphi }\left( x \right),$ respectively.

If $$\underset{x\in \bar G}{\mathop{\max }}\,
\{\left| \varphi \left( x \right)-\overline{\varphi }\left( x \right) \right|\}\le \delta ,$$ then $$\left| u\left( x,t \right)-\overline{u}\left( x,t \right) \right|\le \delta.$$
\end{theorem}
\begin{proof}
The function $\tilde{u}\left( x,t \right)=u\left( x,t \right)-\overline{u}\left( x,t \right)$ satisfies the equation $$\mathcal{D}_{0+,t}^{\alpha }\tilde{u}\left( x,t \right)=\nu\mathcal{D}_{a+,x}^{\beta_1}\mathcal{D}_{a+,x}^{\beta_2}\tilde{u}\left( x,t \right),$$ with initial condition $\tilde{u}\left( x,0 \right)=\varphi \left( x \right)-\overline{\varphi }\left( x \right)$ and Dirichlet condition \eqref{Dir1}. It follows from Theorems \ref{t1} and \ref{t2} that
$$\left| \tilde{u}\left( x,t \right) \right|\le \underset{x\in[a,b]}{\mathop{\max }}\,\{\left| \varphi \left( x \right)-\overline{\varphi }\left( x \right) \right|\}.$$
The result then follows.
\end{proof}

We consider the nonlinear time-space fractional diffusion equation of the form \eqref{1.1},
subject to the initial and boundary conditions \eqref{1.2}, \eqref{Dir1}. We start with the following uniqueness result.

\begin{theorem}\label{4.1}
If $F(x,t,u)$ is nonincreasing with respect to $u$, then the equation \eqref{1.1}, subject to the initial and boundary conditions \eqref{1.2}, \eqref{Dir1}, admits at most one solution $u.$
\end{theorem}
\begin{proof}
Assume that ${{u}_{1}}\left( x,t \right)$ and ${{u}_{2}}\left( x,t \right)$ are two different solutions of the equation \eqref{1.1} subject to initial and boundary conditions \eqref{1.2}, \eqref{Dir1}, and let $${v}\left( x,t \right)={{u}_{1}}\left( x,t \right)-{{u}_{2}}\left( x,t \right).$$ Then ${v}\left( x,t \right)$ satisfies
\begin{equation}\label{4.3}{\mathcal{D}_{0+,t}^{\alpha }{v}\left(x,t \right)}-\nu\mathcal{D}_{a+,x}^{\beta_1}\mathcal{D}_{a+,x}^{\beta_2}{{v}\left( x,t \right)}=F\left( x,t,{{u}_{2}} \right)-F\left( x,t,{{u}_{1}} \right), (x,t)\in\Omega,\end{equation} with homogeneous initial and boundary conditions \eqref{1.2}, \eqref{Dir1}.
Applying the mean value theorem to $F(x,t,u)$ yields
$$F\left( x,t,{{u}_{2}} \right)-F\left( x,t,{{u}_{1}} \right)=\frac{\partial F}{\partial u}\left( {{u}^{*}} \right)\left( {{u}_{2}}-{{u}_{1}} \right)=-\frac{\partial F}{\partial u}\left( {{u}^{*}} \right)v,$$
where ${{u}^{*}}=(1-\mu){{u}_{1}}+\mu{{u}_{2}}$ for some $0\leq \mu \leq1$. Thus,
$${\mathcal{D}_{0+,t}^{\alpha }{v}\left(x,t \right)}-\nu\mathcal{D}_{a+,x}^{\beta_1}\mathcal{D}_{a+,x}^{\beta_2}{{v}\left( x,t \right)}=-\frac{\partial F}{\partial u}\left( {{u}^{*}} \right)v(x,t).$$
Assume by contradiction that $v$ is not identically zero. Then $v$ has either a positive
 maximum or a negative minimum. At a positive maximum $\left( {{x}_{0}},{{t}_{0}} \right)\in {\Omega }$ and as $F(x,t,u)$ is nonincreaing, we have $$\frac{\partial F}{\partial u}\left( {{u}^{*}} \right)\le 0\,\,\,\, \textrm{and} \,\,\,\,-\frac{\partial F}{\partial u}\left( {{u}^{*}} \right)v\left( {{x}_{0}},{{t}_{0}} \right)\ge 0,$$ then $${\mathcal{D}_{0+,t}^{\alpha}{v}\left(x_0,t_0 \right)}-\nu\mathcal{D}_{a+,x}^{\beta_1}\mathcal{D}_{a+,x}^{\beta_2}{{v}\left( x_0,t_0 \right)}\ge 0.$$

By using results of Theorems \ref{t1} and  \ref{t2} for a positive maximum and a negative minimum, respectively, we get ${u}_{1}={u}_{2}.$
\end{proof}

\begin{theorem}\label{4.2}
Let ${{u}_{1}}\left( x,t \right)$ and ${{u}_{2}}\left( x,t \right)$ be two solutions of the equation \eqref{1.1} that satisfy the same boundary condition \eqref{Dir1} and the initial conditions ${{u}_{1}}\left( x,0 \right)={g}_{1}(x)$ and ${{u}_{2}}\left( x,0 \right)={g}_{2}(x),$ $x\in[a,b].$ If $F(x,t,u)$ is nonincreasing with respect to $u$, then it holds that
$${{\left\| {{u}_{1}}\left( x,t \right)-{{u}_{2}}\left( x,t \right) \right\|}_{C\left(\overline{\Omega }\right)}}\leq {{\left\| {{g}_{1}}\left( x \right)-{{g}_{2}}\left( x \right) \right\|}_{C([a,b])}}.$$
\end{theorem}
\begin{proof}
Let ${v}\left( x,t \right)$=${u}_{1}(x,t)-{u}_{2}(x,t)$. Then ${v}\left( x,t \right)$ satisfies the equation
\begin{equation}\label{4.6}{\mathcal{D}_{0+,t}^{\alpha}{v}\left(x,t \right)}-\nu\mathcal{D}_{a+,x}^{\beta_1}\mathcal{D}_{a+,x}^{\beta_2}{{v}\left( x,t \right)}=-\frac{\partial F}{\partial u}\left( {{u}^{*}} \right)v(x,t), (x,t)\in\Omega,
\end{equation}
the initial condition
\begin{equation}\label{4.7}
{{v}\left( x,0 \right)}={{g}_{1}}\left( x \right)-{{g}_{2}}\left(x\right), x \in[a,b],
\end{equation}
and the Dirichlet condition \eqref{Dir1}.
Let
$$
\mathcal{M}={{\left\| {{g}_{1}}\left( x \right)-{{g}_{2}}\left( x \right) \right\|}_{C([a,b])}},
$$
and assume by contradiction that the result of the Theorem \ref{4.2} is not true. That is,
$$
\|u_1-u_2\|_{C(\bar{\Omega})}\nleq \mathcal{M}.
$$
Then $v$ either has a positive maximum at a point $\left( {{x}_{0}},{{t}_{0}} \right)\in {\Omega }$ with $$v\left( {{x}_{0}},{{t}_{0}} \right)=\mathcal{M}_1>\mathcal{M},$$ or it has a negative minimum at a point $\left( {{x}_{0}},{{t}_{0}} \right)\in {\Omega }$ with $$v\left( {{x}_{0}},{{t}_{0}} \right)=\mathcal{M}_2<-\mathcal{M}.$$ If $$v\left( {{x}_{0}},{{t}_{0}} \right)=\mathcal{M}_1>\mathcal{M},$$ using the initial and boundary conditions of $v$, we have $\left( {{x}_{0}},{{t}_{0}} \right)\in{{\bar\Omega }}.$

Analogous proofs of those of Theorem \ref{t1} and Theorem \ref{t2} lead to $\left\| v\left( x,t \right)\right \|\leq{\mathcal{M}};$ this proves the result.
\end{proof}

\section{Time-space fractional pseudo-parabolic equation}

In  this section, we consider the time-space fractional pseudo-parabolic equation with Riemann-Liouville time-fractional derivative
\begin{equation}\label{1.1*}u_t(x,t)=\nu D_{0+,t}^{1-\alpha}\mathcal{D}_{a+,x}^{\beta_1}\mathcal{D}_{a+,x}^{\beta_2} u(x,t)+F\left(x,t\right),\, (x,t)\in (a,b)\times \left(0,T \right]=\Omega,\end{equation}
with Cauchy data
\begin{equation}\label{1.2*}u\left( x, 0\right)=\varphi ( x ),\,x \in[a,b],\end{equation}
and a Dirichlet boundary condition
\begin{equation}\label{1.3*}u\left(a, t\right)=\psi_1(t),\,u\left(b, t\right)=\psi_2(t),\, 0\leq t < T,\end{equation}
where $\alpha, \beta_1, \beta_2\in (0,1],\, 1<\beta_1+\beta_2\leq 2,$ $\nu> 0,$ $-\infty<a<b<+\infty,$ and the functions $F\left( x,t\right),\varphi \left( x \right),\psi_1\left(t \right)$ and $\psi_2\left(t \right)$ are continuous.

\subsection{Maximum principle}
\vskip.3cm
In this subsection, we shall present the maximum principle for the equation \eqref{1.1*}.

\begin{theorem}\label{t1*} Let $u\left( x,t \right)$  satisfy the equation \eqref{1.1*} with initial-boundary conditions \eqref{1.2*}, \eqref{1.3*}. If $F\left( x,t\right)\geq 0$ for $\left( x,t \right)\in \overline{\Omega },$ $\varphi ( x )\geq 0$ for $x \in[a,b],$ $\psi_1(t)\geq 0,\, \psi_2(t)\geq 0$ for $0\leq t\le T,$ then $$u\left( x,t \right)\ge 0\text{ for }\left( x,t \right)\in \overline{\Omega }.$$
\end{theorem}
\begin{proof} For any $\mu\geq 0,$ let $$\tilde{u}(x, t) = u(x, t) + \mu t^{\alpha}.$$
Then
$$\tilde{u}_t(x,t)=u_t(x,t)+\mu\alpha t^{\alpha-1}, (x,t)\in\Omega,$$ $$\tilde{u}(x,0)=u(x,0)=\varphi(x),x\in[a,b],$$ $$\tilde{u}(a,t)=u(a,t)+\mu t^{\alpha}=\psi_1(t)+\mu t^{\alpha}, t\in [0,T].$$
Since $\mathcal{D}_{a+,x}^{\beta_1}\mathcal{D}_{a+,x}^{\beta_2}\tilde{u}(x,t)=\mathcal{D}_{a+,x}^{\beta_1}\mathcal{D}_{a+,x}^{\beta_2} u(x,t),$ we have $$D_{0+,t}^{1-\alpha}\mathcal{D}_{a+,x}^{\beta_1}\mathcal{D}_{a+,x}^{\beta_2}\tilde{u}(x,t)=D_{0+,t}^{1-\alpha}\mathcal{D}_{a+,x}^{\beta_1}\mathcal{D}_{a+,x}^{\beta_2}u(x,t).$$
Hence, $\tilde{u}(x, t)$ satisfies
\begin{align*}\tilde{u}_t(x,t)=\nu D_{0+,t}^{1-\alpha}\mathcal{D}_{a+,x}^{\beta_1}\mathcal{D}_{a+,x}^{\beta_2}\tilde{u}(x,t)+F\left(x,t\right)+\mu\alpha t^{\alpha-1},\, (x,t)\in (a,b)\times \left(0,T \right].\end{align*}

Suppose that there exists some $u(x, t) \in \bar\Omega$ such that $\tilde{u}(x, t) < 0.$ Since $$\tilde{u}(x, t) \geq 0,\, (x, t) \in \{a,b\}\times[0, T]\cup[a,b]\times{0},$$ there is $(x_0, t_0) \in (a,b)\times(0, T]$ such that $\tilde{u}(x_0, t_0)$ is the negative minimum of $\tilde{u}$ over $\bar\Omega.$

It follows from \eqref{02} that

\begin{equation}\label{3.1*}D_{0+,t}^{1-\alpha }\tilde{u}\left(x_0,{t}_{0} \right)\le \frac{t_0^{\alpha-1}}{\Gamma(\alpha)}\tilde{u}\left(x_0, {{t}_{0}} \right) < 0.\end{equation}

Let $w\left( x,t \right)=D_{0+,t}^{1-\alpha }\tilde{u}\left( x,t \right).$
As $\tilde{u}\left( x,t \right)$ is bounded in $\overline{\Omega },$ then we have
\begin{equation}\label{3.2*}
I_{0+,t}^\alpha u\left(x, t \right) = {\rm{
}}\frac{1}{{\Gamma \left( \alpha \right)}}\int\limits_a^t {\left(
t-s \right)^{\alpha  - 1} u\left(x, s \right)} ds \to 0\text{ as }t\to 0.
\end{equation}
It follows from Properties \ref{p5} and \ref{p6} that
$$D_{0+,t}^{\alpha}w\left( x,t \right)=D_{0+,t}^{\alpha}D_{0+,t}^{1-\alpha }\tilde{u}\left( x,t \right)=\tilde{u}_t(x,t).$$
Using Property \ref{p7}, for any $t > 0$, we get
\begin{align*}D_{0+,t}^{1-\alpha }\tilde{u}\left(x,t \right)&=D_{0+,t}^{1-\alpha }u\left( x,t \right)+\mu D_{0+,t}^{1-\alpha}t^{\alpha} \\&= D_{0+,t}^{1-\alpha }u\left( x,t \right)+\mu\frac{\Gamma(\alpha+1)}{\Gamma(2\alpha)} t^{2\alpha-1}.\end{align*}
It follows from Property \ref{p4} that
\begin{equation}\label{4.7}D_{0+,t}^{1-\alpha } u\left(x, t \right)=\frac{1}{\Gamma(\alpha)}\varphi(x)t^{\alpha-1}+\mathcal{D}_{0+,t}^{1-\alpha } u\left(x, t \right)\,\,\,\textrm{for}\,\,\,t>1.\end{equation}
Since the left-hand side of \eqref{4.7} and the first term of the right-hand side of \eqref{4.7} exist, it
follows that the second term on the right-hand side exists and tends to $0$ as $t \rightarrow 0.$ As
$t \rightarrow 0,$ $\varphi(x)t^{\alpha-1}= 0.$ Therefore, $$D_{0+,t}^{1-\alpha}u(x, t) > 0\,\,\, \textrm{when} \,\,\,t = 0.$$ Hence
\begin{align*}w(x,t)=D_{0+,t}^{1-\alpha}\tilde{u}(x,t)=D_{0+,t}^{1-\alpha }u\left( x,t \right)+\mu\frac{\Gamma(\alpha+1)}{\Gamma(2\alpha)} t^{2\alpha-1}.
\end{align*}
Furthermore, it follows from the boundary condition of $\tilde{u}(x, t)$ that
\begin{align*}&w(a,t)=D_t^{1-\alpha}\tilde{u}(a,t)=D_t^{1-\alpha}\psi_1(t)+\mu\frac{\Gamma(\alpha+1)}{\Gamma(2\alpha)} t^{2\alpha-1}\geq 0,\, t\geq 0,\\& w(b,t)=D_{0+,t}^{1-\alpha}\tilde{u}(b,t)=D_{0+,t}^{1-\alpha}\psi_2(t)+\mu\frac{\Gamma(\alpha+1)}{\Gamma(2\alpha)} t^{2\alpha-1}\geq 0,\, t\geq 0.
\end{align*}
Therefore, $w(x, t)$ satisfies the problem
\begin{equation*}\left\{\begin{array}{l}D_{0+,t}^\alpha w(x,t)=\nu \mathcal{D}_{a+,x}^{\beta_1}\mathcal{D}_{a+,x}^{\beta_2}w(x,t)+\tilde{F}(x,t),\,(x,t)\in\Omega, \\ w(x,0)\geq 0, x\in [a,b], \\ w(a,t)\geq 0,\,w(b,t)\geq 0,\,0\leq t\leq T,\end{array}\right.\end{equation*}
where $\tilde{F}(x,t)=F(x,t)+\mu\alpha t^{\alpha-1}.$

From \eqref{3.1*}, it follows that $w(x_0, t_0) < 0.$ Since $w(x, t) \geq 0$ on the boundary, there exists $(x_*, t_*) \in \Omega$ such that $w(x_*, t_*)$ is the negative minimum of function $w(x, t)$ in $\bar\Omega.$  It follows from \eqref{02} that
$$D^{\alpha }_{0+,t} w\left(x_*, {{t}_{*}} \right)\le \frac{1}{\Gamma(1-\alpha)}t_*^{-\alpha}w\left(x_*, {{t}_{*}} \right)<0.$$
Since $w(x_*, t_*)$ is a local minimum, from Proposition \ref{prop1} we obtain $$\mathcal{D}^{\beta_1}_{a+,x}\mathcal{D}^{\beta_2}_{a+,x} w\left(x_*, t_*\right)\geq -\frac{\alpha+\beta-1}{\Gamma(2-\alpha-\beta) }\left(x_* - a\right)^{-\alpha-\beta}w\left(x_*, t_*\right)\geq 0.$$

Therefore at $\left( {{x}_{*}},{{t}_{*}} \right)$, we get $$D_{0+,t}^{\alpha }w\left(x_*, t_* \right)< 0$$ and $$\nu\mathcal{D}^{\beta_1}_{a+,x}\mathcal{D}^{\beta_2}_{a+,x} w\left( {{x}_{*}},{{t}_{*}} \right)+F\left( {{x}_{*}},{{t}_{*}} \right)\ge 0.$$ This contradiction shows that $w\left( x,t \right)\ge 0$ on $\overline{\Omega },$ whereupon  $$u\left( x,t \right)\ge -\mu t^{\alpha}\,\,\, \textrm{on} \,\,\,\overline{\Omega }$$  for any $\mu.$ Since $\mu$ is arbitrary, we have $u(x, t) \geq 0$ on $\bar\Omega.$
\end{proof}
A similar result can be obtained for the non-positivity of the solution $u(x, t)$ by considering $-u(x, t)$ when $F(x,t)\leq 0,$ $\varphi(x) \leq 0,\, \psi_1(t)\leq 0$ and $\psi_1(t) \leq 0.$

\begin{theorem}\label{t2*}
Let $u\left( x,t \right)$  satisfy equation \eqref{1.1*} with initial-boundary conditions \eqref{1.2*}-\eqref{1.3*}. If $F\left( x,t \right)\leq 0$ for $\left( x,t \right)\in \overline{\Omega },$ $\psi_1(t)\leq 0,\, \psi_2(t)\leq 0$ for $0\leq t\le T$ and $\varphi(x)\leq 0$ for $x \in[a,b],$ then $$u\left(x,t \right)\le 0\text{ for }\left( x,t \right)\in \overline{\Omega }.$$ \end{theorem}
Theorem \ref{t1*} and \ref{t2*} leads the following assertions.
\begin{theorem}\label{t1**} Suppose that $u\left( x,t \right)$  satisfies \eqref{1.1*}, \eqref{1.2*}, \eqref{1.3*}. If $F\left( x,t \right)\geq 0$ for $\left( x,t \right)\in \overline{\Omega },$ then $$u\left( x,t \right)\geq \min\limits_{(x,t)\in\bar\Omega}\left\{\varphi(x),\psi_1(t), \psi_2(t)\right\},\,\,\left( x,t \right)\in \overline{\Omega }.$$
\end{theorem}
\begin{proof} Let $m =\min\limits_{(x,t)\in\bar\Omega}\left\{\varphi(x),\psi_1(t), \psi_2(t)\right\}$ and $\tilde{u}(x,t)=u(x,t)-m.$ Then, $$\tilde{u}(x,0)=\varphi(x)-m\geq 0,\,\,x\in[a,b],$$ $$\tilde{u}(a,t)=\psi_1(t)-m\geq 0,\,\tilde{u}(b,t)=\psi_2(t)-m\geq 0,\, 0\leq t\leq T.$$ Since $$\tilde{u}_t(x,t)=u_t(x,t),$$ $$D_{0+,t}^{1-\alpha}\mathcal{D}^{\beta_1}_{a+,x}\mathcal{D}^{\beta_2}_{a+,x}\tilde{u}(x,t)=D_{0+,t}^{1-\alpha}\mathcal{D}^{\beta_1}_{a+,x}\mathcal{D}^{\beta_2}_{a+,x}u(x,t),$$ it follows that $u(x, t)$ satisfies \eqref{1.1*}. Thus, it follows from an argument similar to the proof of Theorem \ref{t1*} that $$\tilde{u}(x,t)\geq 0,\,(x,t)\in\bar\Omega.$$ That is, $$u\left( x,t \right)\geq \min\limits_{(x,t)\in\bar\Omega}\left\{\varphi(x),\psi_1(t), \psi_2(t)\right\},\,\,\left( x,t \right)\in \overline{\Omega }.$$ The theorem \ref{t1**} is proved.
\end{proof}
A similar result can be obtained for the nonpositivity of the solution $u(x, t)$ by considering $-u(x, t).$
\begin{theorem} Suppose that $u\left( x,t \right)$  satisfies \eqref{1.1*}, \eqref{1.2*}, \eqref{1.3*}. If $F\left( x,t \right)\leq 0$ for $\left( x,t \right)\in \overline{\Omega },$ then $$u\left( x,t \right)\leq \max\limits_{(x,t)\in\bar\Omega}\left\{\varphi(x),\psi_1(t), \psi_2(t)\right\},\,\,\left( x,t \right)\in \overline{\Omega }.$$
\end{theorem}

\subsection{Uniqueness results}
\vskip.3cm
The maximum principle for the time-space fractional equation \eqref{1.1*} can be used to prove the uniqueness of a solution.
\begin{theorem}\label{t3*}
The problem \eqref{1.1*}, \eqref{1.2*}, \eqref{1.3*} has at most one solution.
\end{theorem}
\begin{proof}
Let ${{u}_{1}}\left( x,t \right)$ and ${{u}_{2}}\left( x,t \right)$ be two solutions of the initial-boundary value problem \eqref{1.1*}, \eqref{1.2*}, \eqref{1.3*} and $\hat{u}(x,t)={{u}_{1}}\left( x,t \right)-{{u}_{2}}\left( x,t \right)$. Then,
$$\hat{u}_t\left( x,t \right)=\nu D_{0+,t}^{1-\alpha}\mathcal{D}^{\beta_1}_{a+,x}\mathcal{D}^{\beta_2}_{a+,x}\hat{u}\left( x,t \right),$$
with homogeneous initial and boundary conditions \eqref{1.2*}, \eqref{1.3*} for $\hat{u}\left( x,t \right)$. It follows from Theorems \ref{t1*} and \ref{t2*} that $\hat{u}\left( x,t \right)=0$ on $\overline{\Omega}.$ Consequently ${{u}_{1}}\left( x,t \right)={{u}_{2}}\left( x,t \right).$ The result then follows.\end{proof}
Theorems \ref{t1*} and \ref{t2*} can be used to show that a solution $u\left( x,t \right)$ of the problem \eqref{1.1*}, \eqref{1.2*}, \eqref{1.3*} depends continuously on the initial data $\varphi \left( x \right).$

\begin{theorem}\label{t4*}
Suppose $u\left( x,t \right)$ and $\bar{u}\left( x,t \right)$ are the solutions of the equation \eqref{1.1*} that satisfy the same boundary condition \eqref{1.3*} and the initial conditions $u\left( x,0 \right)=\varphi(x)$ and $\bar{u}\left( x,0 \right)=\bar{\varphi}(x),$ $x\in[a,b].$ If $\underset{x\in [a,b]}{\mathop{\max }}\,\{\left| \varphi \left( x \right)-\bar{\varphi }\left( x \right) \right|\}\le \delta,$ then $$\left| u\left( x,t \right)-\bar{u}\left( x,t \right) \right|\le \delta.$$
\end{theorem}
\begin{proof}
The function $\tilde{u}\left( x,t \right)=u\left( x,t \right)-\bar{u}\left( x,t \right)$ satisfies the equation $$\tilde{u}_t\left( x,t \right)=\nu D_{0+,t}^{1-\alpha}\mathcal{D}^{\beta_1}_{a+,x}\mathcal{D}^{\beta_2}_{a+,x}\tilde{u}\left( x,t \right),$$ with initial condition $\tilde{u}\left(x, 1\right)=\varphi \left( x \right)-\bar{\varphi }\left( x \right)$ and boundary condition \eqref{1.3*}. It follows from Theorems \ref{t1*} and \ref{t2*} that
$$\left| \tilde{u}\left( x,t \right) \right|\le \underset{x\in[a,b]}{\mathop{\max }}\,\{\left| \varphi \left( x \right)-\bar{\varphi }\left( x \right) \right|\}.$$
The result then follows.
\end{proof}

\begin{theorem}\label{4.1*}
If $F(x,t,u)$ is nonincreasing with respect to $u$, then the nonlinear time-space fractional pseudo-parabolic equation \begin{equation}\label{1.1**}u_t(x,t)=\nu D_{0+,t}^{1-\alpha}\mathcal{D}_{a+,x}^{\beta_1}\mathcal{D}_{a+,x}^{\beta_2} u(x,t)+F\left(x,t,u\right),\, (x,t)\in \Omega,\end{equation} subject to the initial and boundary conditions \eqref{1.2*}, \eqref{1.3*} admits at most one solution.
\end{theorem}
\begin{theorem}\label{4.2*}
If ${{u}_{1}}\left( x,t \right)$ and ${{u}_{2}}\left( x,t \right)$ are two solutions of the equation \eqref{1.1**} that satisfy the same boundary condition \eqref{1.3*} and the initial conditions $${{u}_{1}}\left( x,0 \right)={g}_{1}(x)$$ and $${{u}_{2}}\left( x,0 \right)={g}_{2}(x), x\in[a,b]$$ and if $F(x,t,u)$ is nonincreasing with respect to $u$, then it holds that
$${{\left\| {{u}_{1}}\left( x,t \right)-{{u}_{2}}\left( x,t \right) \right\|}_{C\left(\overline{\Omega }\right)}}\leq {{\left\| {{g}_{1}}\left( x \right)-{{g}_{2}}\left( x \right) \right\|}_{C([a,b])}}.$$
\end{theorem}
Theorems \ref{4.1*} and \ref{4.2*} are proved similarly as Theorems \ref{4.1} and \ref{4.2} in Subsection \ref{subsection2}.

\section{Fractional elliptic equation}
In this section, we consider an elliptic equation with a sequential Caputo derivative in a multidimensional parallelepiped
\begin{equation}\label{6.1}\begin{split}\Delta_x u(x)&+\sum\limits_{j=1}^n a_j(x) \mathcal{D}_{a+,x_{j}}^{\alpha}\mathcal{D}_{a+,x_{j}}^{\beta}u(x)\\&+\sum\limits_{j=1}^n b_j(x) \frac{\partial u}{\partial x_j}(x)+\sum\limits_{j=1}^n c_{j}(x) \mathcal{D}_{a+,x_{j}}^{\gamma}u(x)\\&+d(x)u(x)=F(x),\,x=(x_1,...,x_n)\in \prod\limits_{j=1}^n (p_j, q_j)=\Omega,\end{split}\end{equation}
where $0<\alpha, \beta\leq 1,$ $1<\alpha+\beta\leq 2, 0<\gamma\leq 1,$ $-\infty<p_j<q<j<+\infty,$ $a_j(x), b_j(x), c(x)$ and $F(x)$ are given functions, and $$\Delta_x=\frac{\partial^2}{\partial x^2_1}+...+\frac{\partial^2}{\partial x^2_n}=\sum\limits_{j=1}^n \frac{\partial^2}{\partial x^2_j}.$$

\subsection{Weak and strong maximum principle}
\vskip.3cm
We start with a weak maximum principle.
\begin{theorem}\label{t6.1} Let a function $u(x)$ satisfy the equation \eqref{6.1} and $a_j(x)\geq 0, c_j(x)<0,\,d(x)\leq 0,\,x\in\bar\Omega.$ If $F(x)\geq 0,$ then the inequality \begin{equation}\label{6.2}\max\limits_{x\in\bar\Omega} u(x)\leq\max\limits_{x\in\partial\Omega} \{u(x),0\}\end{equation} holds true, where $\partial\Omega$ is the boundary of $\Omega.$
\end{theorem}
\begin{proof} Let us assume that the inequality \eqref{6.2} does not hold true under the conditions that are formulated in Theorem \ref{t6.1}, i.e. that the function $u(x)$ attains its positive maximum, say $M > 0$ at a point $x^*=(x_1^*,...,x_n^*) \in \Omega.$

Because $$d(x^*)\leq 0, \frac{\partial u}{\partial x_j}(x^*)=0\,\,\, \textrm{and} \,\,\,\frac{\partial^2 u}{\partial x^2_j}(x^*)\leq 0,$$ we first get the inequality
$$\Delta_x u(x^*)+c(x^*)u(x^*)\leq 0.$$
Then, it follows from \eqref{01} that
$$\mathcal{D}_{a+,x_j}^{\gamma}u \left( {{x}^*} \right)\ge \frac{1}{\Gamma(1-\gamma)}\left(x_j^*-c_j\right)^{-\gamma} u\left( x^* \right)>0.$$
As $c_{j}(x^*)<0,$ then
$$\sum\limits_{j=1}^n c_{j}(x^*) \mathcal{D}_{a+,x_{j}}^{\gamma}u(x^*)<0.$$
By Proposition \ref{prop1}, we have
$$\mathcal{D}_{a+,x_j}^{\alpha}\mathcal{D}_{a+,x_j}^{\beta} u\left(x^*\right)\leq 0.$$
The last two inequalities lead to the inequality
$$\Delta_x u(x^*)+\sum\limits_{j=1}^n a_j(x^*) \mathcal{D}_{a+,x_{j}}^{\alpha}\mathcal{D}_{a+,x_{j}}^{\beta}u(x^*)+\sum\limits_{j=1}^n c_{j}(x^*) \mathcal{D}_{a+,x_{j}}^{\gamma}u(x^*)+d(x^*)u(x^*)<0$$ that contradicts the following one: \begin{align*}\Delta_x u(x)&+\sum\limits_{j=1}^n a_j(x) \mathcal{D}_{a+,x_{j}}^{\alpha}\mathcal{D}_{a+,x_{j}}^{\beta}u(x)+\sum\limits_{j=1}^n b_j(x) \frac{\partial u}{\partial x_j}(x)\\&+\sum\limits_{j=1}^n c_{j}(x) \mathcal{D}_{a+,x_{j}}^{\gamma}u(x)+d(x)u(x)\geq 0,\,x\in\Omega\end{align*} of Theorem \ref{t6.1}. The theorem is proved.
\end{proof}
The following theorem is proved similarly.
\begin{theorem}\label{t6.2} Let a function $u(x)$ satisfy the equation \eqref{6.1} and $a_j(x)\geq 0, c_j(x)>0,\,d(x)\leq 0,\,x\in\bar\Omega.$ If $F(x)\leq 0,$ then the inequality \begin{equation}\label{6.3}\min\limits_{x\in\bar\Omega} u(x)\geq\min\limits_{x\in\partial\Omega} \{u(x),0\}\end{equation} holds true.
\end{theorem}
\begin{remark}\label{rem1} In the proof of the weak maximum principle, we have in fact deduced a statement that is stronger than the inequality \eqref{6.2}, namely, we proved that a function u that fulfills the conditions of Theorem \ref{t6.1} cannot attain its positive maximum at a point $x^* \in \Omega.$\end{remark}
The statement of Remark \ref{rem1} is now employed to derive a strong maximum principle for the elliptic equation \eqref{6.1}.

\begin{theorem}\label{t6.3} Let a function $u(x)$ satisfy the homogeneous elliptic equation \eqref{6.1} and $d(x) \leq 0, x\in\Omega.$
If the function u attains its maximum and its minimum at some points that belong to $\Omega,$ then it is a constant, more precisely $$u(x) =0,\,x\in \Omega.$$\end{theorem}

\begin{proof} Indeed, according to Remark \ref{rem1}, $$u(x) \leq 0, x\in\bar\Omega$$ for a function $u(x)$ that attains its maximum at a point $x^*\in\Omega.$
Now let us consider the function $-u(x)$ that satisfies the homogeneous equation \eqref{6.1} and possesses a maximum at the minimum point of $u(x)$ and thus at a point that belongs to $\Omega.$ The
maximum of $-u(x)$ cannot be positive according to Remark \ref{rem1} and we get the inequality $$-u(x) \leq 0, x\in\Omega.$$ The two last inequalities assert the statement of Theorem \ref{t6.3}.
\end{proof}
\subsection{Applications of the maximum principles}
In this section, we add the boundary condition
\begin{equation}\label{6.4}u(x)=\varphi(x),\,x\in\partial\Omega,\end{equation} to the elliptic equation \eqref{6.1}.
The following result is a direct consequence of the weak maximum principle

\begin{theorem} Let $F(x),a_j(x), b_j(x), c_j(x), \varphi(x)$ and $d(x)\leq 0$ be smooth functions. Then the boundary-value problem \eqref{6.1}, \eqref{6.4} admits at most one solution $u(x).$
\end{theorem}

The following two theorems follow directly from Theorem \ref{t6.1} and Theorem \ref{t6.2}.
\begin{theorem}\label{t6.4} Let $u(x)$ fulfill the equation \eqref{6.1} and $d(x) \leq 0, x\in\Omega.$ If $u(x)$ satisfies the boundary condition \eqref{6.4} and $\varphi(x) \geq 0, x\in \partial\Omega,$ then $$u(x) \geq 0, x\in\bar\Omega.$$
\end{theorem}

\begin{theorem}\label{t6.5} Let $u(x)$ is the solution of elliptic equation \eqref{6.1} and $d(x) \leq 0, x\in\Omega.$ If $u(x)$ satisfies the boundary condition \eqref{6.4} and $\varphi(x) \leq 0, x\in \partial\Omega,$ then $$u(x) \leq 0, x\in\bar\Omega.$$
\end{theorem}

\section{Fractional Laplace equation in cylindrical domain}
In this section, we consider the following fractional Laplace equation
\begin{equation}\label{FL1}
-\mathcal{D}_{a+,x}^{\alpha}\mathcal{D}_{a+,x}^{\beta}u(x,y)+\left(-\Delta\right)^\delta_yu(x,y)=f(x,y),\, (x,y)\in (a,b)\times\Omega=\Sigma,
\end{equation}
where $\Omega\subset \mathbb{R}^N, N\geq 1$ is a bounded domain, $1<\alpha+\beta\leq 2,$ $\delta\in(0,1)$ and $(-\Delta)^\delta$ is the
regional fractional Laplace operator defined as follows (see \cite{Vald})
$$\left(-\Delta\right)^\delta u(y)=c_{N,\delta}\,\textrm{P.V.}\int_{\Omega}\frac{u(y)-u(\xi)}{|y-\xi|^{N+2\delta}}d\xi,$$ with normalizing constant $c_{N,\delta}=\frac{\delta 2^{2\delta}\Gamma\left(\frac{N+2\delta}{2}\right)}{\pi^{N/2}\Gamma(1-\delta)}.$

We consider the fractional Laplace equation \eqref{FL1} with boundary conditions
\begin{equation}\label{FL2}u(x,y)=0,\, x\in [a,b],\, y\in \mathbb{R}^N\backslash\Omega=\partial\Omega,\end{equation}
\begin{equation}\label{FL3}u(a,y)=\varphi_1(y),\,u(b,y)=\varphi_2(y),\,y\in\Omega,\end{equation} where $\varphi_1$ and $\varphi_2$ are continuous functions.
\begin{theorem}\label{thFL1} Let $u$ be a continuous solution of equation \eqref{FL1} and $\mathcal{D}_{a+,x}^{\beta}u\in C(\Sigma),$ $\mathcal{D}_{a+,x}^{\alpha}\mathcal{D}_{a+,x}^{\beta}u\in C(\Sigma)$ and $\left(-\Delta\right)^\delta_yu\in C(\Sigma).$ If $f \geq 0$ in $\bar\Sigma,$ and $\varphi_1\geq 0, \varphi_2\geq 0$ in $\Omega,$ then $u\geq 0$ in $\bar\Sigma.$
\end{theorem}
\begin{proof}Let us argue by contradiction. Assume $u < 0$ somewhere in $\bar\Sigma,$ then there exists $(x^*, y^*) \in\bar\Sigma$ such that $u(x^*,y^*)=\min\limits_{(x,y)\in\bar\Sigma}u<0.$ Since
$u \geq 0$ in $\{a\}\times\partial\Omega\cup\{b\}\times\partial\Omega=\Gamma,$ we have $(x^*, y^*)\notin \Gamma.$

Since $(x^*, y^*)$ is a minimum, then according to Proposition \ref{prop1} we have $$\mathcal{D}_{a+,x}^{\alpha}\mathcal{D}_{a+,x}^{\beta}u(x^*,y^*)\geq 0.$$ As $\left(-\Delta\right)^\delta_yu\in C(\Sigma)$ and $u$ attains its minimum at $(x^*, y^*),$ we have
$$\left(-\Delta\right)^\delta u(x^*,y^*)=c_{N,\delta}\,\textrm{P.V.}\int_{\Omega}\frac{u(x^*,y^*)-u(x^*,\xi)}{|y^*-\xi|^{N+2\delta}}d\xi\leq 0.$$ If $\left(-\Delta\right)^\delta u(x^*,y^*)=0,$ then $u(x^*,\cdot)=0,$ which is a contradiction with $u(x^*,y^*)<0,$ therefore $\left(-\Delta\right)^\delta u(x^*,y^*)<0.$ But, then
$$0\leq f(x^*,y^*)=-\mathcal{D}_{a+,x}^{\alpha}\mathcal{D}_{a+,x}^{\beta}u(x^*,y^*)+\left(-\Delta\right)^\delta_yu(x^*,y^*)<0.$$ A contradiction. Therefore, $u\geq 0$ in $\bar\Sigma.$ The theorem is proved.
\end{proof}
A similar result can be obtained for the nonpositivity of the solution $u$ by considering $-u.$
\begin{theorem}\label{thFL2}Let $u$ be a continuous solution of equation \eqref{FL1} and $\mathcal{D}_{a+,x}^{\beta}u\in C(\Sigma),$ $\mathcal{D}_{a+,x}^{\alpha}\mathcal{D}_{a+,x}^{\beta}u\in C(\Sigma)$ and $\left(-\Delta\right)^\delta_yu\in C(\Sigma).$ If $f \leq 0$ in $\bar\Sigma,$ and $\varphi_1\leq 0, \varphi_2\leq 0$ in $\Omega,$ then $u\leq 0$ in $\bar\Sigma.$
\end{theorem}
Theorem \ref{thFL1} and \ref{thFL2} leads the following results.
\begin{theorem}Let $u$ be a solution of problem \eqref{FL1}, \eqref{FL2}, \eqref{FL3} and let $\mathcal{D}_{a+,x}^{\beta}u\in C(\Sigma),$ $\mathcal{D}_{a+,x}^{\alpha}\mathcal{D}_{a+,x}^{\beta}u\in C(\Sigma)$ and $\left(-\Delta\right)^\delta_yu\in C(\Sigma).$ Then we have the following two assertions
\begin{description}
  \item[(A)] If $u$ satisfies the inequality $$-\mathcal{D}_{a+,x}^{\alpha}\mathcal{D}_{a+,x}^{\beta}u(x,y)+\left(-\Delta\right)^\delta_yu(x,y)\leq 0\,\,\, \textrm{in} \,\,\,\Sigma,$$ then $$\max\limits_{(x,y)\in\bar\Sigma}u=\max\limits_{(x,y)\in\Gamma}u.$$
  \item[(B)] If $u$ satisfies the inequality $$-\mathcal{D}_{a+,x}^{\alpha}\mathcal{D}_{a+,x}^{\beta}u(x,y)+\left(-\Delta\right)^\delta_yu(x,y)\geq 0\,\,\, \textrm{in} \,\,\,\Sigma,$$ then $$\min\limits_{(x,y)\in\bar\Sigma}u=\min\limits_{(x,y)\in\Gamma}u.$$
\end{description}
\end{theorem}
\begin{theorem}Let $F(x,y),$ $\varphi_1(y)$ and $\varphi_2(y)$ be smooth functions. Then the Dirichlet problem \eqref{FL2}-\eqref{FL3} for the fractional Laplace equation \eqref{FL1} possesses at most one solution $u.$
\end{theorem}

\section*{Conclusion}
We obtained the extremum principle for the sequential Caputo derivative of order $(1,2]$ and presented some of its applications. This statement answered positively to Luchko's question about maximum principle for the space and time-space fractional partial differential equations.\\
The results  of this article are:
\begin{itemize}
  \item the estimates for the sequential Caputo fractional derivative of order $(1,2]$ at the extremum points is proved;
  \item comparison principles for the linear and non-linear fractional ordinary differential equations is obtained;
  \item maximum and minimum principles for the time-space fractional diffusion and pseudo-parabolic equations with Caputo and Riemann-Liouville time-fractional derivatives are derived;
  \item uniqueness of solution and continuous dependence of a solution on the initial conditions of the initial-boundary problems for the nonlinear time-space fractional diffusion and pseudo-parabolic equations are proved;
  \item maximum and minimum principles for the elliptic equations in multidimensional parallelepiped and cylindrical domains with sequential Caputo derivative are derived.
\end{itemize}

\section*{Acknowledgements} The research of Kirane is supported by NAAM research group, University of King Abdulaziz, Jeddah. The research of Torebek is financially supported in parts by the FWO Odysseus 1 grant G.0H94.18N: Analysis and Partial Differential Equations and by the grant No.AP08052046 from the Ministry of Science and Education of the Republic of Kazakhstan. No new data was collected or generated during the course of research


\begin{thebibliography}{ABGM15}
\bibitem[AV2019]{Vald} N. Abatangelo, E. Valdinoci. {\it Getting acquainted with the fractional Laplacian}. Springer INdAM Ser.,
Springer, Cham. (2019).

\bibitem[Al2012]{Al-Refai} M. Al-Refai. On the fractional derivatives at extreme points. {\it Electronic Journal of Qualitative Theory of Differential Equations.} {\bf 2012}, No. 55 (2012), 1-5.

\bibitem[AlL2014]{Al-Refai-Luchko1} M. Al-Refai, Y. Luchko. Maximum principle for the fractional diffusion equations with the Riemann-Liouville fractional derivative and its applications. {\it Fractional Calculus and Applied Analysis.} {\bf 17}, No. 2 (2014), 483-498.

\bibitem[AlL2015]{Al-Refai-Luchko2} M. Al-Refai, Y. Luchko. Maximum principle for the multi-term time-fractional diffusion equations with the Riemann-Liouville fractional derivatives. {\it Applied Mathematics and Computation.} {\bf 257}, (2015), 40-51.

\bibitem[AlL2017]{Al-Refai-Ab} M. Al-Refai, T. Abdeljawad. Analysis of the fractional diffusion equations with fractional derivative of non-singular kernel. {\it Advances in Difference Equations}. {\bf 2017}, (2017), 1-12.

\bibitem[Al2018]{Al-Refai2} M. Al-Refai. Comparison principles for differential equations involving Caputo fractional derivative with Mittag-Leffler non-singular kernel. {\it Electronic Journal of Differential Equations}. {\bf 2018}, (2018), 1-10.

\bibitem[BKT2018]{KiraneBT} M. Borikhanov, M. Kirane, B. T. Torebek. Maximum principle and its application for the nonlinear time-fractional diffusion equations with Cauchy-Dirichlet conditions. {\it Applied Mathematics Letters}. {\bf 81}, (2018), 14-20.

\bibitem[BT2018]{Borikhanov} M. Borikhanov, B. T. Torebek. Maximum principle and its application for the sub-diffusion equations with Caputo-Fabrizio fractional derivative. {\it Matematicheskii Zhurnal}. {\bf 18}, No. 1 (2018), 43-52.

\bibitem[CS2014]{Cabre} X. Cabr\'{e}, Y. Sire. Nonlinear equations for fractional Laplacians, I: Regularity, maximum principle and Hamiltonian estimates. {\it Annales de l'Institut Henri Poincar\'{e} C, Analyse non liin\'{e}aire}. {\bf 31}, (2014), 23-53.

\bibitem[CKZ2017]{CaoKong} L. Cao, H. Kong, Sh.-D. Zeng. Maximum principles for time-fractional Caputo-Katugampola diffusion equations. {\it Journal of Nonlinear Sciences and Applications.} {\bf 10}, (2017), 2257-2267.

\bibitem[CDDS2011]{Capella} A. Capella, J. D\'{a}vila, L. Dupaigne, Y. Sire. Regularity of radial extremal solutions for some non-local semilinear equations. {\it Communications in Partial Differential Equations.} {\bf 36}, No. 8 (2011), 1353-1384

\bibitem[CL2016]{Chan} C. Y. Chan, H. T. Liu. A maximum principle for fractional diffusion equations. {\it Quarterly of Applied Mathematics.} {\bf 74}, No. 3 (2016), 421-427.

\bibitem[CHL2017]{Cheng} T. Cheng, G. Huang, C. Li. The maximum principles for fractional Laplacian equations and their applications. {\it Communications in Contemporary Mathematics}. {\bf 19}, No. 6 (2017), 1750018-1-1750018-12.

\bibitem[DQ2017]{Del} L. M. Del Pezzo, A. A. Quaas. A Hopf's lemma and a strong minimum principle for the fractional p-Laplacian. {\it Journal of Differential Equations}. {\bf 263}, No. 1 (2017), 765-778.

\bibitem[JL2016]{JiaLi} J. Jia, K. Li. Maximum principles for a time-space fractional diffusion equation. {\it Applied Mathematics Letters}. {\bf 62}, (2016), 23-28.

\bibitem[KST2006]{Kilbas} A. A. Kilbas, H. M. Srivastava and J. J. Trujillo. \emph{Theory and Applications of Fractional Differential Equations}. North-Holland Mathematics Studies. (2006).

\bibitem[KT2018]{KiraneTorebek} M. Kirane, B. T. Torebek. Extremum principle for the Hadamard derivatives and its application to nonlinear fractional partial differential equations. {\it Fractional Calculus and Applied Analysis.} {\bf 22}, No. 2 (2019) 358-378.

\bibitem[LZB2016]{LiuZeng} Z. Liu, S. Zeng, Y. Bai. Maximum principles for multi-term space-time variable-order fractional diffusion equations and their applications. {\it Fractional Calculus and Applied Analysis.} {\bf 19}, No. 1 (2016), 188-211.

\bibitem[L2009a]{Luchko1} Y. Luchko. Maximum principle for the generalized time-fractional diffusion equation. {\it Journal of Mathematical Analysis and Applications.} {\bf 351}, (2009), 218-223.

\bibitem[L2010]{Luchko2} Y. Luchko. Some uniqueness and existence results for the initial boundary-value problems for the generalized time-fractional diffusion equation. {\it Computers and Mathematics with Applications}. {\bf 59}, (2010), 1766-1772.

\bibitem[L2011a]{Luchko3}  Y. Luchko. Initial-boundary-value problems for the generalized multiterm time-fractional diffusion equation. {\it Journal of Mathematical Analysis and Applications.} {\bf 374}, (2011), 538-548.

\bibitem[L2009b]{Luchko4} Y. Luchko. Boundary value problems for the generalized time-fractional diffusion equation of distributed order. {\it Fractional Calculus and Applied Analysis} {\bf 12}, No. 4 (2009), 409-422.

\bibitem[L2011b]{Luchko5} Y. Luchko. Maximum principle and its application for the time-fractional diffusion equations. {\it Fractional Calculus and Applied Analysis} {\bf 14}, No. 1 (2011), 110-124.

\bibitem[L2012]{Luchko6} Y. Luchko, Initial-boundary-value problems for the one-dimensional time-fractional diffusion equation. {\it Fractional Calculus and Applied Analysis} {\bf 15}, No. 1 (2012) 141-160.


\bibitem[LY2017]{LuchkoYa} Y. Luchko, M. Yamamoto. On the maximum principle for a time-fractional diffusion equation. {\it Fractional Calculus and Applied Analysis}. {\bf 20}, No. 5 (2017), 1131-1145.

\bibitem[MR1993]{MillerRoss} K.S. Miller, B. Ross, An  introduction  to  the  fractional  calculus and fractional differential equations. A Wiley-Interscience Publication. John Wiley and Sons,  Inc.,  New  York,  1993.

\bibitem[N2010]{Nieto} J. J. Nieto. Maximum principles for fractional differential equations derived from Mittag–Leffler functions. \emph{Applied Mathematics Letters}. {\bf 23}, (2010), 1248-1251.

\bibitem[P1999]{Podlubny} I. Podlubny,  Fractional  differential  equations.  An  introduction  to  fractional derivatives, fractional differential equations, to methods of their solution and some of  their  applications.  Mathematics  in  Science  and  Engineering,   198. AcademicPress,  Inc.,  San  Diego,  CA, 1999.

\bibitem[SKM93]{SKM87}
S.~G.~Samko, A.~A.~Kilbas, and O.~I.~Marichev.
\newblock {\it Fractional Integrals and Derivatives, Theory and Applications.}
\newblock Gordon and Breach, Amsterdam, 1993.

\bibitem[SZ2009]{ShiZhang} A. Shi and S. Zhang, Upper and lower solutions method and a fractional differential equation boundary value problem, {\it Electronic Journal of Qualitative Theory of
Differential Equations}. {\bf 30}, (2009), 1--13.

\bibitem[TT16]{TT16} N.~Tokmagambetov, T.~B.~Torebek.
\newblock Fractional Analogue of Sturm--Liouville Operator.
\newblock {\it Documenta Math.}, {\bf 21}, (2016), 1503--1514.

\bibitem[TT14]{TT14} B.Kh. Turmetov, B.T. Torebek, On solvability of some boundary value problems for a fractional analogue of the Helmholtz equation, {\it New York Journal of Mathematics}, {\bf 20}, (2014), 1237--1251.

\bibitem[YLAT2014]{YeLiu} H. Ye, F. Liu, V. Anh, I. Turner. Maximum principle and numerical method for the multi-term time–space Riesz-Caputo fractional differential equations. {\it Applied Mathematics and Computation}. {\bf 227}, (2014), 531-540.

\bibitem[YLAT14]{YLAT14} H. Ye, F. Liu, V. Anh, I. Turner. Maximum principle and numerical method for the multi-term time-space Riesz-Caputo fracional differential equations. {\it Appl. Math. Comp.} {\bf 227}, (2014), 531-540.

\bibitem[ZAW2017]{Ahmad} L. Zhang, B. Ahmad, G. Wang. Analysis and application of diffusion equations involving a new fractional derivative without singular kernel. {\it Electronic Journal of Differential Equations}. {\bf 2017}, (2017), 1-6.

\end{thebibliography}
\end{document}